\documentclass[12pt]{amsart}
\usepackage{fullpage,verbatim,amssymb}
\usepackage{hyperref, bbm}
\usepackage{soul}
\usepackage[active]{srcltx}
\usepackage[usenames,dvipsnames]{color}
\hypersetup{colorlinks=true,linkcolor=blue,citecolor=magenta}

\makeatletter
\let\@@pmod\bmod
\DeclareRobustCommand{\bmod}{\@ifstar\@pmods\@@pmod}
\def\@pmods#1{\mkern4mu({\operator@font mod}\mkern 6mu#1)}
\makeatother

\definecolor{blue}{rgb}{0,0,1}
\definecolor{red}{rgb}{1,0,0}
\definecolor{green}{rgb}{0,.6,.2}
\definecolor{purple}{rgb}{1,0,1}

\long\def\red#1\endred{\textcolor{red}{#1}}
\long\def\blue#1\endblue{\textcolor{blue}{#1}}
\long\def\purple#1\endpurple{\textcolor{purple}{ #1}}
\long\def\green#1\endgreen{\textcolor{green}{#1}}

\newcommand{\ph}{\varphi}

\newcommand{\im}{\text{Im}}

\newcommand{\scrL}{\mathcal{L}}

\newcommand{\scrF}{\mathcal{F}}

\newcommand{\Z}{\mathbb{Z}}
\newcommand{\N}{\mathbb{N}}

\newcommand{\R}{\mathbb{R}}
\newcommand{\C}{\mathbb{C}}

\newcommand{\HH}{\mathbb{H}}

\DeclareMathOperator{\SL}{SL}

\DeclareMathOperator{\lcm}{{lcm}}

\newcommand{\sm}{\left(\begin{smallmatrix}}
\newcommand{\esm}{\end{smallmatrix}\right)}
\newcommand{\bpm}{\begin{pmatrix}}
\newcommand{\ebpm}{\end{pmatrix}}

\newtheorem{theorem}{Theorem}

\newtheorem{proposition}[theorem]{Proposition}
\newtheorem{corollary}[theorem]{Corollary}

\theoremstyle{remark}

\newtheorem{question}[theorem]{Question}
\newtheorem*{rmk}{Remark}

\numberwithin{theorem}{section}
\numberwithin{equation}{section}

\title{Analogues of the Bol operator for half-integral weight weakly holomorphic modular forms}
\author{Nikolaos Diamantis} 
\address{University of Nottingham}
\email{nikolaos.diamantis@nottingham.ac.uk}
\author{Min Lee} 
\address{University of Bristol}
\email{min.lee@bristol.ac.uk}
\author{Larry Rolen}
\address{Vanderbilt University}
\email{larry.rolen@vanderbilt.edu}

\begin{document}

\begin{abstract}
We define an analogue of the Bol operator on spaces of weakly holomorphic modular forms of half-integral weight. We establish its main properties and relation with other objects.
\end{abstract}
\maketitle

\section{Introduction} The classical ``Bol operator'' has proved a very fruitful tool in various aspects of the theory of modular forms. It provides one of the ways to address the difficulty that the derivative of a modular form is typically not modular (see \S 5 of \cite{Zagier123} for a excellent discussion of this problem). Among the applications of the Bol operator, we only point to two: Firstly, the theory of period polynomials \cite{KnoppMasterpiece} and, through it, the critical values of $L$-functions, bases of spaces of cusp forms etc. Secondly, the theory of harmonic Maass forms \cite[Ch. 5]{book} to which the Bol operator plays a fundamental role, not least because, together with the ``xi-operator" (see \eqref{xi}), they uniquely determine the harmonic Maass form.

We outline its construction in the setting we will most often be using, namely that of weakly holomorphic modular forms. 
For $N \in \N$ and $k \in \Z$, let $M_k^!(N)$ denote that space of weakly holomorphic modular forms of weight $k$ for $\Gamma_0(N)$, i.e. modular forms for which the holomorphicity condition is relaxed to include functions with poles at the cusps.
Then we set 
\begin{equation}\label{e:Bol_def}
D^{k-1}:=(2 \pi i)^{1-k} \frac{d^{k-1}}{d z^{k-1}}.
\end{equation}
This induces a map from $M_{2-k}^!(N)$ to $M_k^!(N)$ given, at the level of Fourier expansions, by
\begin{equation}\label{BolsId}D^{k-1}\left ( \sum_{n\gg-\infty}a_nq^n \right )=\sum_{n\gg-\infty}a_n n^{k-1}q^n.
\end{equation} 

This Bol operator commutes with the Hecke operators and with the Fricke involution $W_N:=\left ( \begin{smallmatrix} 0 & -1/\sqrt{N} \\
\sqrt{N} & 0 \end{smallmatrix} \right )$. 
It can be expressed as an iterated Maass raising operator \cite[Lemma~5.3]{book} 
and forms a companion to  the ``shadow operator'' on the space $H_{2-k}(\Gamma_0(N))$ of harmonic Maass forms 
\begin{equation}\label{xi}
\xi_{2-k}:=2iv^k\overline{\frac{\partial}{\partial\overline{z}}}\colon H_{2-k}(\Gamma_0(N))\rightarrow  S_k(\Gamma_0(N)), 
\end{equation}
where $v=\im(z)$. Here $S_k(\Gamma_0(N))$ stands for the space of cusp forms of weight $k$ and level $N.$
The interplay between $D^{k-1}$ and $\xi_{2-k}$ is fundamental for the  theory of harmonic Maass forms and, 
in particular, the study of mock modular forms. 
Specifically, harmonic Maass forms canonically split into two pieces, which are in turn annihilated each by one of these two operators. Thus, the two pieces of a harmonic Maass form can both be uniquely determined via positive weight (weakly holomorphic) modular forms by using both operators.

Given the importance of the ``Bol operator'' hinted above, it is natural to seek analogues in the space $M^!_k(N, \chi)$ of \emph{half-integral} weight $k$ weakly holomorphic modular forms for $\Gamma_0(N)$ and character $\chi$. In contrast to the shadow operator which exists for $k$ half-integral and behaves exactly as in integral weight, there is, as yet, no ``companion'' operator to $\xi_{2-k}$. Finding one has been a long-standing aim among researchers in the area. 
 
Results hinting in this direction have been given, however.
In \cite{BGK} a very interesting map is constructed which sends weight $2-k$ harmonic Maass forms to weight $k$ weakly holomorphic cusp forms for $k$ half-integral in a fashion that parallels the Bol operator. 
The construction is based on the Zagier lifts and appears in the context of Shintani lifts from integral weight weakly holomorphic modular forms to half-integral weight weakly holomorphic modular forms. 
Since the main aim was to ensure that the maps involved in the definition of the Shintani lift are Hecke invariant, the Bol-style map of \cite{BGK} was, in fact, a family of maps on certain individual subspaces of the space of weight $2-k$ weak Maass forms. 
However, unlike the classical Bol operator, these maps do not have a simple action on Fourier expansions, and while they have been put to good use to study $L$-values, they are in some sense more mysterious.

Against this background, in this note we investigate three questions of increasing strength and specificity
\begin{question}
\label{Question11}
Can one build an explicit analogue of the operator in \eqref{BolsId} for the entire space of half-integral weight forms?
\end{question}
A stronger form of this question is
\begin{question}
\label{Question12}
Let $k \in \frac{1}{2}+\N,$ $N \in \N$ and a Dirichlet character $\psi$ mod $N$. 
Does there exist a linear map from $M^!_{2-k}(N, \psi)$ to $M^!_{k}(N', \psi')$, for some $N' \in \N $ and a character $\psi' \bmod N'$, sending each
\begin{equation}\label{FourExmg}
f(z)=\sum_{n\ge -n_0}c_n q^n \in M^!_{2-k}(N, \psi)
\end{equation}
to a $f_1 \in M^!_{k}(N', \psi')$ which has the form
\begin{equation}\label{Q2}f_1(z)=\sum_{n\ge -n_0}\left ( c_n \ell (n)  n^{k-1}+ \text{``lower order terms''} \right )q^n
\end{equation}
  for an explicit, bounded map $\ell \colon \Z \to \C$ independent of $f$ and, as ``lower order terms'', some linear combinations of $\{c_{-n_0}, \dots, c_{n-1}\}$ with coefficients independent of $f$? 
\end{question}
An even stronger version of the question, dispenses with the ``lower order terms'' in the $n$-th Fourier coefficient of \eqref{Q2}
\begin{question}
\label{Question13}
Let $k \in \frac{1}{2}+\N,$ $N \in \N$ and a Dirichlet character $\psi$ mod $N$. 
Does there exist a linear map from $M^!_{2-k}(N, \psi)$ to $M^!_{k}(N', \psi')$, for some $N' \in \N $ and a character $\psi' \bmod N'$, sending each $f \in M^!_{2-k}(N, \psi)$ with Fourier expansion \eqref{FourExmg}
to a $f_1 \in M^!_{k}(N', \psi')$ of form
\begin{equation}\label{Q3}f_1(z)=\sum_{n\ge -n_0} c_n \ell (n)  n^{k-1}q^n
\end{equation}
  for an explicit, bounded map $\ell \colon \Z \to \C$ independent of $f$?
\end{question}

We will give an affirmative answer to Questions \ref{Question11} and \ref{Question12} by defining a family of Bol-style maps on the entire space $M^!_{2-k}(N, \psi)$ yielding elements of $M^!_{k}(N, \psi')$ with a Fourier expansion of the form \eqref{Q2}. 
For the construction, the relation between the theta functions 
$$
\theta_0(z):=\sum_{n\geq1}\psi_0(n)q^{n^2} 
\qquad \text{and} \, \, \, \theta_1(z):=\sum_{n\geq1}n\psi_1(n)q^{n^2},$$
for suitable characters $\psi_0, \psi_1$, is used as a prototype of a Bol-style operator of weight $1/2$ (see Prop. \ref{Qu13}) and this is reflected in the structure of the formula for our operators. Indeed, in addition to addressing Questions \ref{Question11} and \ref{Question12}, they map $\theta_0$ to $\theta_1$. Further, our operators are derived by a process reminiscent of (group-)conjugating an integral weight Bol operator by elements of the algebra generated by $\theta_0, \theta_1$. 

The use of theta functions in our construction is similar in spirit with the special case of the Shimura lift that originated with Selberg and, later on, extended by Cipra and Hansen-Naqvi \cite{C1, HN}. 
In their setting too, the theta function is used to ``complete'' the degree to an integer. A further similarity is that both their Shimura lift and our Bol-style operator are explicitly identified and, in the negative direction, that neither their Shimura lift nor our operator are compatible with the Hecke action. 
Finally, there is no basis to ask whether their version of the Shimura lift is compatible with our Bol-style operator and the classical Bol operator, because their lift, by construction, is only definable for positive weights.

It is unclear whether the answer to Question \ref{Question13} is affirmative and the main supportive evidence is, on the one hand, the analogy with the integral weight case and, on the other, the special case of the pair $\theta_0, \theta_1$ in weights $1/2$ and $3/2$. A positive answer would be important, not just because the resulting construction will be simpler but, mainly, because a function such as \eqref{Q2} will be more likely to be compatible with the Hecke action and with the analogue of the Shimura lift for weakly holomorphic modular forms. In the last section of the note, we propose an approach towards this question which is based on the direct and converse theorems proved in our recent work \cite{DLRR}.  

\section*{Acknowledgements} The first author is partially supported by EPSRC grant EP/S032460/1. The second author was supported by Royal Society University Research Fellowship ``Automorphic forms, L-functions and trace formulas''. The third author is grateful for support from a grant from the Simons Foundation (853830, LR), support from a Dean’s Faculty Fellowship from Vanderbilt University. Part of the work was done while the first and third author were visiting Max Planck Institute for Mathematics in Bonn, whose hospitality they acknowledge.

\section{Background.} 
\subsection{Basic notation} We recall the slash action $|_k$ of $\SL_2(\R)$ on smooth functions $f\colon \HH \to \C$ on the complex upper half-plane $\HH$, in the cases $k \in \Z$ and $k \in \frac{1}{2}+\Z$:

{\bf $\bullet$ $k \in \Z$.} We consider the action $|_k$ of $\SL_2(\R)$ on smooth functions $f\colon \HH \to \C$ on the complex upper half-plane $\HH$, given by 
\begin{equation}\label{e:slashk_def}
(f|_k\gamma)(z):= (cz+d)^{-k} f(\gamma z), \qquad \text{for $\gamma=\bpm a & b \\  c & d \ebpm \in$ SL$_2(\R)$}.
\end{equation} 
Here $\gamma z = \frac{az+b}{cz+d}$ is the M\"obius transformation.  

{\bf $\bullet$ $k\in \frac{1}{2}+\Z$.} Here and throughout, we set the implied logarithm to equal its principal branch so that $-\pi <$arg$(z) \le \pi$.   
If $\left ( \frac{c}{d} \right )$ denotes the Kronecker symbol, we set, for an odd integer $d$,
\begin{equation}
\epsilon_d:=\begin{cases} 1 & \text{ if } d\equiv 1\bmod{4}, \\
i & \text{ if } d\equiv 3\bmod{4}, 
\end{cases} 
\end{equation}
so that $\epsilon_d^2 = \left(\frac{-1}{d}\right)$. 
We define the action $|_k$ of $\Gamma_0(N)$, for $4|N$, on smooth functions $f\colon \HH \to \C$ as follows: 
\begin{equation}\label{e:slashk_halfint}
(f|_k\gamma)(z):= 
\left ( \frac{c}{d} \right ) \epsilon_d^{2k} (cz+d)^{-k} f(\gamma z) \qquad \text{ for all } \gamma=\bpm * & * \\  c & d \ebpm  \in \Gamma_0(N).
\end{equation}

Let $W_M=\sm 0 & -\sqrt{M}^{-1} \\ \sqrt{M} & 0\esm$ for $M\in \mathbb{N}$. 
For both $k \in \frac12\Z$ we have
\begin{equation}\label{WNinC}
(f|_k W_M)(z)
=f(W_Mz) (\sqrt M z)^{-k}. 
\end{equation}
Note that we define the action $|_k W_M$ by the same formula in both the integral and half-integral weight cases. 

For each $N \in \mathbb N,$ $k \in \frac{1}{2}\Z$, and Dirichlet character $\psi \bmod N$, we denote by $M_k(N, \psi)$ (resp. $M^!_k(N, \psi)$) the space of holomorphic (resp. weakly holomorphic) modular forms for $\Gamma_0(N)$ and character $\psi$. The subspace of $M_k(N, \psi)$ (resp. $M^!_k(N, \psi)$ consisting of forms such that the constant term of its Fourier expansion at each cusp vanishes is denoted by $S_k(N, \psi)$ (resp. $S^!_k(N, \psi)$ and called the space of holomorphic (resp. weakly holomorphic) cusp forms for $\Gamma_0(N)$ and character $\psi$. The absence of $\psi$ from the notation means that the implied character is the trivial one.

\subsection{$L$-series}
Following \cite{DLRR}, we associate an $L$-series to each element of the above spaces. 
As discussed in \cite{DLRR}, in the case of holomorphic cusp forms, these $L$-series are equivalent with the classical $L$-series. 
 
First, for each $f: \HH \to \C$, holomorphic in $\HH$ with a Fourier expansion of the form 
\begin{equation}\label{FourExmgGene}
f(z)=\sum_{n\ge -n_0}c_n e^{2 \pi i n z/M}
\end{equation}
for some $M \in \N$ and $n_0 \in \Z.$, we let $\mathcal F_f$ be the set of piecewise smooth complex functions $\varphi$ on $\R$ such that the series
\[\sum_{\substack{n \ge -n_0}} |c_n| (\scrL|\varphi|)\left (2 \pi n/M\right )\] 
converges, where $(\scrL \varphi)(s):=\int_0^{\infty} e^{-s t} \varphi(t)dt $ is the Laplace transform. 

Let now $ k \in \frac{1}{2} \Z$, $N \in \N$ and a Dirichlet character $\psi \bmod N.$ If the Fourier expansion of $f \in M^!_{2-k}(N, \psi)$ has the form \eqref{FourExmg} with $M=1$, we define the $L$-series of $f$ to be the map
$L_f\colon \mathcal F_f \to \C$  given by
\[
L_f(\ph):=\sum_{n \ge -n_0} c_n (\scrL \ph)(2 \pi n).
\]

We further consider the twists $f_\chi(z)$ given by 
\[
f_\chi(z): = \sum_{n\geq -n_0} c_n\tau_{\bar{\chi}}(n) e^{2\pi n\frac{z}{D}},
\]where for a Dirichlet character $\chi$ modulo $D$ and an $n \in \Z$, the {\it generalized Gauss sum} is 
\[
\tau_{\chi}(n):=\sum_{u \bmod D} \chi(u)e^{2 \pi i n\frac{u}{D}}.
\]
The $L$-function of $f_{\chi}$ is defined by
\[
L_f(\chi, \ph)=L_{f_{\chi}}(\ph):=\sum_{n \ge -n_0} c_n\tau_{\bar{\chi}}(n) (\scrL \ph)(2 \pi n/D)
\]
for each $\ph \in \scrF_{f_{\chi}}$.

One of the main advantages of this distributional-type setup is that it was used in  \cite{DLRR} to prove a Weil-type converse theorem.
To state the direct and converse theorems for our $L$-series, we set, for each $\varphi: \R^+\to \C$
\begin{equation}\label{e:testft_kWM_def}
(\varphi|_k W_M)(x):= \varphi((Mx)^{-1}) (Mx)^{-k}. 
\end{equation}

We recall the following theorem from \cite{DLRR}.
\begin{theorem}\label{DThalf}
Fix $k\in \frac{1}{2}\Z$. 
Let $N\in \N$ and let $\psi$ be a Dirichlet character modulo $N$.
When $k\in \frac{1}{2}+\Z$, assume that $4|N$. 
Suppose that $f$ is an element of $M^!_k(N, \psi)$ with expansion \eqref{FourExmg} and that $\chi$ is a 
character modulo $D$ with $(D, N)=1$. 
Set $g:=f|_kW_{N}$
and 
\begin{equation}\label{e:scrF_fg}
\scrF_{f, g} := \bigcap\limits_{\chi \bmod D}\left\{\varphi\in \scrF_{f_{\chi}}   
\;:\; \varphi|_{2-k} W_N \in 
\scrF_{g_{\chi}}  \right\}.  
\end{equation}
Then $\scrF_{f, g}\neq \{0\}$ and we have the following functional equations. 
For each $\varphi\in \scrF_{f, g}$, if $k\in \Z$, 
\begin{equation}\label{e:FEN0}
L_{f}(\chi, \ph) =i^k
\frac{\chi(-N) \psi(D)}{N^{k/2-1}} L_{g}(\bar \chi, \ph|_{2-k}W_N),
\end{equation}
For each $\varphi\in \scrF_{f, g}$, if $k\in \frac{1}{2}+\Z$, 
\begin{equation}\label{FENhalf}
L_{f}(\chi, \ph)  = i^k \psi_D(-1)^{k-\frac{1}{2}} \psi_D(N) 
\frac{\chi(-N) \psi(D)}{\epsilon_D N^{k/2-1}} L_{g}(\bar{\chi}\psi_D, \ph|_{2-k}W_N).
\end{equation}
Here $\psi_D(u) = \left(\frac{u}{D}\right)$ is the real Dirichlet character modulo $D$, given by the Kronecker symbol.
\end{theorem}

Note that the factor in \eqref{FENhalf} differs from \cite{DLRR} due to our different normalisation of $f|_kW_N$. 
We now recall the converse of Theorem
\ref{DThalf} from \cite{DLRR}. 
\begin{theorem}\label{thm:CT1}
Let $N$ be a positive integer and $\psi$ a Dirichlet character modulo $N$. 
For $j\in \{1, 2\}$ and some integer $n_0$, let $(a_j(n))_{n\geq -n_0}$ be a sequence of complex numbers such that 
$a_j(n)=O\left(e^{C\sqrt{|n|}}\right)$ as $|n|\to \infty$ for some constant $C>0$. 
Define holomorphic functions $f_j\colon \HH\to \C$ by
$
f_j(z) := \sum_{n\geq -n_0} a_j(n) e^{2\pi i n z}.
$

For all $D\in \{1, 2, \ldots, N^2-1\}$, $\gcd(D, N)=1$, Dirichlet character $\chi$ modulo $D$ and any smooth, compactly supported $\varphi: \R_+ \to \C$, assume that
\[
L_{f_1}(\chi, \varphi)
= i^k \frac{\chi(-N)\psi(D)}{N^{\frac{k}{2}-1}} 
L_{f_2}( \overline{\chi}, \varphi|_{2-k}W_N) 
\]
if $k\in \Z$, and 
\[
L_{f_{1}}(\chi, \varphi)
= i^k\psi_D(-1)^{k-\frac{1}{2}} \psi_D(N) 
\frac{\chi(-N) \psi(D)}{\epsilon_D N^{\frac{k}{2}-1}} 
L_{f_2}(\overline{\chi}\psi_D, \varphi|_{2-k}W_N)
\]
if $k\in \frac{1}{2}+\Z$. 
Then, the function $f_1$ is a weakly holomorphic modular form with weight $k$ and Nebentypus character $\psi$ for $\Gamma_0(N)$ and $f_2=f_1|_k W_N$. 
\end{theorem}

\section{Analogues of the Bol operator}
\subsection{Theta series of weight $1/2$ and $3/2$}
Before beginning our main construction, we first recall the classical unary theta functions \cite{Sh}.
Let $N_0, N_1$ two positive integers. 
Fix an even $\psi_0 \bmod N_0$ and 
an odd character $\psi_1 \bmod N_1$.
Set 
\begin{equation} \label{t0}
\theta_0(z):=\sum_{n\geq0}\psi_0(n)q^{n^2} 
\end{equation}
and 
\begin{equation} \label{t1}
\theta_1(z):=\sum_{n\geq1}n\psi_1(n)q^{n^2}. 
\end{equation}
For computational convenience, we take $\psi_0(0)$ to be $1/2$, $\psi_0$ is the trivial character. 
By \cite[Section 2]{Sh}, we have that $\theta_0$
is a modular form of weight $1/2$, level $4N_0^2$ and character $\psi_0$. 
Also, $\theta_1$ is a modular form of weight $3/2$, level $4N_1^2$ and character $\psi_1 \left ( \frac{-1}{\cdot}\right )$. 
If, in addition, $\psi_0$ (resp. $\psi_1$) are real and primitive, we have the transformation equations
\begin{align}
& \theta_0|_{1/2}W_{4N_0^2}=(iN_0)^{-1/2} \tau(\psi_0)\theta_0; \label{WN0}
\\ & \theta_1|_{3/2}W_{4N_1^2}=-(iN_1)^{-1/2} \tau(\psi_1)\theta_1, \label{WN1}
\end{align}
where, for $i=0, 1$, 
\[
\tau(\psi_i):=\sum_{u \bmod N_i} \psi_i(u)e^{2 \pi i \frac{u}{N_i}}.
\]

\begin{rmk}
The primitivity of $\psi_0, \psi_1$ is needed only for \eqref{WN0}, \eqref{WN1} not for the modularity of $\theta_0, \theta_1$.
\end{rmk}

Recall that $\psi_0$ is a real, even, primitive Dirichlet character modulo $N_0$. 
In the next proposition, we will answer Question \ref{Question13} in the case of the subspace $M_{1/2}(4N_0^2, \psi_0)$ of $M_{1/2}^!(4N_0^2, \psi_0)$. Specifically, we will define a map such as the one posited in Question \ref{Question13}. 
Any $f\in M_{1/2}(4N_0^2, \psi_0)$ with Fourier expansion $f(z) = \sum_{n\geq 0} a(n) q^n$, is mapped to the function
\begin{equation}
(\delta^{\frac{3}{2}-1} f)(z) 
:= \sum_{n\geq 0} a(n) \ell(n) n^{\frac{3}{2}-1} q^n 
\end{equation}
with $\ell: \N_0 \to \R$ given by
\begin{equation}
\ell(n): = \begin{cases}
\left(\frac{-1}{\sqrt{n}} \right) & \text{ when } \sqrt{n} \in \mathbb{N}, \\
0 & \text{ otherwise. } 
\end{cases} 
\end{equation}
Now we will prove that the image of the map $\delta^{\frac{3}{2}-1}$ is contained in $M_{3/2}(64N_0^2, \psi_0)$.
\begin{proposition}\label{Qu13}  
The answer to Question \ref{Question13} is positive in $M_{1/2}(4N_0^2, \psi_0)$, for $\psi_0 \mod N_0$ real, even and primitive.
\end{proposition}
\begin{proof}
First note that, by \cite[Theorem A]{SeS}, a basis of $M_{1/2}(4 N_0^2, \psi_0)$ consists of the series
\[\theta_{\psi, t}(z):=\sum_{n=0}^{\infty} \psi(n)q^{tn^2}\]
for $(\psi, t)$ such that $t \in \N$,  $\psi$ is even, primitive with conductor $r$, $4r^2t|4N_0^2$ and $\psi_0(m) = \psi(m)\chi_t(m)$ whenever $(m, 4N_0^2)=1$. 
Here $\chi_t$ is the primitive character of order less than equal $2$ corresponding to the field extension $\mathbb{Q}(t^{\frac{1}{2}})$ over $\mathbb{Q}$ (\cite[p.30, \S1.2]{SeS}). 
When $t$ is square we set $\chi_t=1$. 
As above, by convention, $\psi(0)=1/2$ if $\psi$ is the trivial character.


For a non-negative integer $n$, we define
\begin{equation}
a_t(n) = \begin{cases} \psi\left(\sqrt{\frac{n}{t}}\right) & \text{ when } t\mid n \text{ and } \sqrt{\frac{n}{t}}\in \mathbb{N}, \\
0 & \text{ otherwise. } 
\end{cases} 
\end{equation}
Then 
\begin{equation}
\theta_{\psi, t}(z) = \sum_{m=0}^\infty a_t(m) q^{m}. 
\end{equation}
We have
\begin{equation}
(\delta^{\frac{3}{2}-1} \theta_{\psi, t})(z) 
= \sum_{m\geq 0} a_t(m) \ell(m) m^{\frac{3}{2}-1} q^{m}
= \sum_{n\geq 0} \psi(n) \ell(tn^2) (tn^2)^{\frac{3}{2}-1} q^{tn^2}.
\end{equation}
Note that $\ell(tn^2)=0$ unless $t$ is square. 
So $\delta^{\frac{3}{2}-1} \theta_{\psi, t}=0$ unless $t$ is square. 
If $t$ is square, then $\psi_0=\psi$ and thus $r=N_0$, which implies that $t=1$. Therefore,
\begin{equation}
\delta^{\frac{3}{2}-1} \theta_{\psi, t}(z)
=0 
\end{equation}
unless $(\psi, t) = (\psi_0, 1)$. 
When $(\psi, t)=(\psi_0, 1)$, comparing with \eqref{t1}, we get
\begin{equation}
(\delta^{\frac{3}{2}-1} \theta_{\psi_0,1})(z)
= \sum_{n\geq 0} \psi_0(n)\left(\frac{-1}{n}\right) n q^{n^2}
= \theta_1(z) 
\end{equation}
with $\psi_1 = \psi_0\left(\frac{-1}{\cdot}\right)$. Since, $\psi_1$ is a Dirichlet character modulo $4 N_0$,  $\theta_1$ is a weight $3/2$ modular form of level $4(4 N_0)^2$ and character $\psi_1 \left(\frac{-1}{\cdot}\right)=\psi_0.$ 
Therefore, the assignment
$$\sum_{n\geq0}a_t(n)q^{n} \longrightarrow 
\sum_{n\geq0}a_t(n)\ell(n) n^{\frac{3}{2}-1}q^{n}
$$
induces a linear map from $M_{1/2}(4N_0^2, \psi_0)$ to $M_{3/2}(64 N_0^2, \psi_0)$, confirming the assertion of Question \ref{Question13}.
\end{proof}
This instance of a positive answer to Question \ref{Question13} is special because it concerns forms of moderate growth which can occur only if both weights involved ($k$ and $2-k$) are positive. 
This happens only if $k=1/2$ or $3/2$ and therefore, for more general half-integral weights one must by necessity consider weakly holomorphic forms. 
However, the relation between $\theta_0$ and $\theta_1$ will be used as the basis for the general weight case in the next subsection.

\subsection{The main construction}\label{MainCon}
Throughout subsections \ref{MainCon}-\ref{FE}, we fix a real, primitive, even $\psi_0 \bmod N_0$ and 
a real, primitive odd character $\psi_1 \bmod N_1$ such that
$\theta_0(\cdot, \psi_0)$ and $\theta_1(\cdot, \psi_1)$ have no zeros in $\HH$. 
Further fix an integer $a$. 
Then, for each smooth function $f: \HH \to \C$, set \begin{equation}\label{def}
\delta^{k-1}_a(f):=\theta_0^{3a-2}(z) \theta_1^{1-a}(z) D^{k-\frac{3}{2}}(\theta_0^{1-3a}\theta_1^a f), 
\end{equation}
where $D^{k-\frac{3}{2}}
$ is the usual Bol's operator given in \eqref{e:Bol_def}. 
Here, note that $k-\frac{3}{2} \in \N_0$. 

\begin{theorem} 
Set $N:=\lcm(4N_0^2, 4N_1^2)$.
For each $\gamma= \sm * & * \\ c & d \esm\in \Gamma_0(N)$, we have
$$\delta^{k-1}_a(f|_{2-k}\gamma)=\left ( \frac{-1}{d} \right )\frac{\psi_0(d)}{\psi_1(d)}\delta^{k-1}_a(f)|_k\gamma.$$
Suppose further that $N_0=N_1$. Then, with $N=\lcm(4N_0^2, 4N_1^2)=4N_0^2$, we have
$$\delta^{k-1}_a(f|_{2-k}W_N)=-\frac{\tau(\psi_0)}{\tau(\psi_1)}\delta^{k-1}_a(f)|_{k}W_N.$$
\end{theorem}
\begin{proof}
Since $\theta_0$ (resp. $\theta_1$) has weight $1/2$ (resp. $3/2$) and character $\psi_0$ (resp. $\psi_1 \cdot \left( \frac{-1}{\cdot} \right )$) for $\Gamma_0(4N_0^2) \supset \Gamma_0(N)$ (resp. $\Gamma_0(4N_1^2) \supset \Gamma_0(N)$), we have
\begin{multline}\delta^{k-1}_a(f|_{2-k}\gamma)=\theta_0^{3a-2}(z) \theta_1^{1-a}(z) D^{k-\frac{3}{2}}\left((\theta_0|_{\frac{1}{2}} \gamma)^{1-3a}(\theta_1|_{\frac{3}{2}}\gamma)^a f|_{2-k}\gamma \right )\bar \psi_0(d)^{1-3a} \bar \psi_1(d)^{a} \left ( \frac{-1}{d}\right )^a \\
=\theta_0^{3a-2}(z) \theta_1^{1-a}(z) D^{k-\frac{3}{2}}\left ((f \theta_0^{1-3a} \theta_1^a)(\gamma z) (cz+d)^{-\frac{1}{2}-(2-k)}\right )  \epsilon_d^{1+2(2-k)} \bar \psi_0(d)^{1-3a} \bar \psi_1(d)^{a} \left ( \frac{-1}{d}\right )^a.
\end{multline}
The function to which $D^{k-3/2}$ has been applied equals $(f \theta_0^{1-3a} \theta_1^a)|_{2-(k-1/2)}\gamma$ (note that the weight is integral). Then we apply the standard Bol's identity followed by the modularity of $\theta_0$ and $\theta_1$. We obtain
\begin{multline}
\theta_0^{3a-2}(z) \theta_1^{1-a}(z) \epsilon_d^{1+2(2-k)} (D^{k-\frac{3}{2}}(f \theta_0^{1-3a} \theta_1^a )|_{k-\frac12}\gamma)(z) \, \bar \psi_0(d)^{1-3a} \bar \psi_1(d)^{a} \left ( \frac{-1}{d}\right )^a \\
=(\theta_0|_{1/2}\gamma)(z)^{3a-2} (\theta_1|_{3/2})(z)^{1-a} \epsilon_d^{1-2k} (D^{k-\frac{3}{2}}(f \theta_0^{1-3a} \theta_1^a )|_{k-\frac12}\gamma)(z) \, \bar \psi_0(d)^{-1} \bar \psi_1(d) \left ( \frac{-1}{d}\right ) \\
=\theta_0(\gamma z)^{3a-2} \theta_1^{1-a}(\gamma z)  D^{k-\frac{3}{2}}(f \theta_0^{1-3a} \theta_1^a )(\gamma z) (cz+d)^{-k} \left ( \frac{c}{d}\right ) \epsilon_d^{-2k}\frac{\psi_0(d)}{\psi_1(d)}.
\end{multline}
In the last equality we used the identity $\epsilon_d^2=\left ( \frac{-1}{d}\right )$. Using the same identity and the definition of the action $|_k$ (with $k$ half-integral), we can simplify to deduce the first identity of the theorem.

Suppose now that $N_0=N_1$. Then \eqref{WN0}, \eqref{WN1} imply that
$$\theta_0^{1-3a} \theta_1^{a} (f|_{2-k}W_N)= \frac{(-1)^a (iN_0)^{1/2-a}}{\tau(\psi_0)^{1-3a}\tau(\psi_1)^a}(\theta_0^{1-3a} \theta_1^{a} f)|_{2-(k-1/2)}W_N.$$
Bol's identity, followed by an application of \eqref{WN0} and \eqref{WN1}, imply that
\begin{multline}
\delta^{k-1}_a(f|_{2-k} W_N)
=\frac{(-1)^a (iN_0)^{1/2-a}}{\tau(\psi_0)^{1-3a}\tau(\psi_1)^a}
\theta_0^{3a-2} \theta_1^{1-a} (D^{k-3/2}(\theta_0^{1-3a} \theta_1^{a} f))|_{k-1/2}W_N
\\ = \frac{(-1)^a (iN_0)^{1/2-a}}{\tau(\psi_0)^{1-3a}\tau(\psi_1)^a}
\frac{(-1)^{1-a}(iN_0)^{a-\frac{1}{2}}}{\tau(\psi_0)^{3a-2}\tau(\psi_1)^{1-a}}
\\ \times (\theta_0|_{1/2}W_N)^{3a-2} (\theta_1|_{3/2}W_N)^{1-a}
D^{k-3/2}(\theta_0^{1-3a} \theta_1^{a} f)|_{k-1/2}W_N
\end{multline}
which, after simplification, implies the second identity of the theorem.
\end{proof}
With the notation of the theorem, we see that the answer to Question \ref{Question11} is positive: 
\begin{corollary}\label{AnsQu11}
Let $f$ be a weakly holomorphic modular form of weight $2-k \in \frac{1}{2}-\N_0$, level $N$ and character $\psi$. Then 
$\delta^{k-1}_a(f)$ is a weakly holomorphic modular form of weight $k$, level $N$ and character $d \to \psi(d) \left ( \frac{-1}{d} \right )\frac{\psi_1(d)}{\psi_0(d)}$.
In particular,
$$\delta^{1/2}_a(\theta_0)=\theta_1,$$
\end{corollary}
Note that since $\theta_0$ has no zeros in $\HH$, 
$\delta^{k-1}(f)$ is well-defined and gives a weakly holomorphic form.

In our construction, the parameter $a$ is assumed to be integer. However, $\delta_a^{k-1}$ can be defined for other values too and, in some cases, it can be shown to coincide with other well-known operators. We will discuss one such example.
 
 We first note that the equation defining $\delta_a^{k-1}$ in \eqref{def} gives a well defined function when $a \in \mathbb Q$. We also recall the definition of Rankin-Cohen bracket in the form given, e.g. in \cite{Za}, which includes the case of half-integral weights. For $n \in \mathbb N_0$ and modular forms $f$ and $g$ of level $N$, weights $k$, $\ell$ respectively, and characters $\chi, \psi$ respectively, we set
 $$[f, g]_n=\sum_{j=0}^n (-1)^{n-j}\binom{n}{j}\frac{\Gamma(k+n) \Gamma(\ell+n)}{\Gamma(k+j)\Gamma(\ell+n-j)}f^{(j)}g^{(n-j)}$$
 where $f^{(j)}$ denotes the $j$-th derivative of $f.$ The function $[f, g]_n$ is a modular form of weight $k+\ell+2n$ and character $\chi \psi.$ Then we have
\begin{corollary}\label{c:sel}
Let $f$ be a weakly holomorphic modular form of weight $-1$, level $N$ and character $\psi$. Set
\begin{equation}\label{sel}
F(z)=f(4z)\theta_0(z).
\end{equation} Then 
$$\delta^{3/2}_{2/3}(F)=\frac{3}{\pi i}[\theta_1, f(4 \cdot)]_1$$
\end{corollary}
\begin{proof} This follows by a direct calculation.
\end{proof}
\begin{rmk}
We have not been able to find so simple a relation of our Bol-style operator with Rankin-Cohen brackets for general weights and values of the parameter $a$. However, the structure of the half-integral weight forms that are the subject of Corollary \ref{c:sel}, exhibits some interesting similarity with that of modular forms that can be lifted to integral weight modular forms according to Selberg's version of the Shimura lift. This will be discussed in the next subsection.
\end{rmk}

\subsection{Selberg's version of the Shimura lift}\label{SelSh}
The construction of the operator $\delta_a^{k-1}$ has quite a few analogies with Selberg's version of a Shimura-type lift (see \cite{C1}). We recall its simplest case and then point out those analogies:

\emph{Let $k$ be an a positive even integer and let $f(z)=\sum_{n \ge 1} a(n)q^n \in S_{k}(1)$ be a normalised eigenform for all Hecke operators. Set
$$F(z):=f(4z) \theta_0(z) \in S_{k+1/2}(4).$$
Then 
$$S(F)(z):=f(z)^2-2^{k-1}f(2z)^2$$
belongs to $
S_{2 k}(2).$}

The most essential analogy of Selberg's construction with the construction of $\delta_a^{k-1}$ is that it uses a fixed half-integral weight form (namely $\theta_0$, as in our construction) in order to translate the situation into the more familiar context of integral weight forms. Furthermore, this ``translation" is based on a multiplication with the theta function. 

As in our construction, Selberg's lift is not Hecke-invariant. Furthermore, since his construction is intrinsically built on holomorphic cusp forms, it can only be carried out in the case of positive weight and therefore we cannot check if $\delta_a^{k-1}$ is compatible with $D^{2k-2}$ via Selberg's lift.

\subsection{Fourier expansion}\label{FE} We will compute the Fourier expansion of $\delta^{k-\frac{3}{2}}_a(f)$ in the special case $a=0$ which will allow us to answer Question \ref{Question12} 

In addition to the assumptions of the previous subsections, we assume that the Dirichlet character $\psi_0$ associated to the theta series $\theta_0$ is also non-trivial.
Recall that
\begin{equation}
    \label{D^k}
    D^{k-\frac{3}{2}} \left (\sum_{m \gg -\infty} a_nq^n \right )=\sum_{m \gg -\infty} n^{k-\frac{3}{2}}a_nq^n.
\end{equation}
Since $\theta_0$ is non-trivial, we note that $\theta_1/\theta_0^2$ has a Fourier expansion of the form
\begin{equation}\label{Fexp}
    \frac{\theta_1(z)}{\theta_0(z)^2}=\sum_{n=-1}^{\infty} a_nq^n \qquad \text{with $a_{-1}=1.$}
\end{equation}
We then have:
\begin{proposition}
With notation as above, let $$f(z)=\sum_{n=-n_0}^{\infty}c_n q^n$$
be a weakly holomorphic modular form of weight $2-k \in \frac{1}{2}-\N$, level $N=\lcm(4N_0^2, 4N_1^2)$ and character $\psi$. Then, the Fourier expansion of $\delta_0^{k-1}(f)$ is given by
$$\delta_0^{k-1}(f)(z)=\sum_{n=-n_0}^{\infty}q^n \left(\sum_{l=-n_0}^n c_{l} \left ( \sum_{m=1}^{n+1-l} (l+m)^{k-\frac{3}{2}}\psi_0(\sqrt{m})a_{n-l-m}\right )\right )$$
where, $\psi_0(\sqrt{m})=\psi_0(m_1)$, if $m=m_1^2$ ($m \in \N$), and $0$ otherwise.
 \end{proposition}
\begin{proof}
With the Fourier expansions of $\theta_0$, $f$ and \eqref{Fexp} we see that 
$$\delta_0^{k-1}(f)(z)=\left ( \sum_{n=-1}^{\infty} a_nq^n \right ) D^{k-\frac{3}{2}} \left ( \sum_{m=1-n_0}^{\infty}q^{m} \left ( \sum_{l=1}^{m+n_0}\psi_0(\sqrt{l})c_{m-l}\right )\right ).$$
With \eqref{D^k} followed by the change of variables $n+m \to n,$ we deduce
\begin{multline}\label{Fexp1}
\delta_0^{k-1}(f)(z)=\sum_{n=-n_0}^{\infty} q^n \left ( \sum_{m=1-n_0}^{n+1}m^{k-\frac{3}{2}}a_{n-m} \left ( \sum_{l=1}^{m+n_0}\psi_0(\sqrt{l})c_{m-l}\right )\right )\\
=\sum_{n=-n_0}^{\infty} q^n \left ( \sum_{m=1-n_0}^{n+1}m^{k-\frac{3}{2}}a_{n-m} \left ( \sum_{l=-n_0}^{m-1}\psi_0(\sqrt{m-l})c_{l}\right )\right )
\end{multline}
 which, by an interchange of the inner sums, equals
 \begin{equation*}
\sum_{n=-n_0}^{\infty} q^n  \sum_{l=-n_0}^{n}\sum_{m=l+1}^{n+1}m^{k-\frac{3}{2}}a_{n-m}  \psi_0(\sqrt{m-l})c_{l}
\end{equation*}
The change of variables $m \to l+m$ in the last sum implies the result.
\end{proof}
This proposition, together with Cor. \ref{AnsQu11}, allows us to answer Question \ref{Question12} positively:
\begin{corollary}
Let $k \in \frac{1}{2}+\N,$ $N \in \N$ and a Dirichlet character $\psi$ mod $N$. Set
$\ell(n):=(n+1)^{k-\frac{3}{2}}/n^{k-1}$ if $n \neq 0$ and $\ell(0):=0$.
Then $\delta_0^{k-1}$ is a linear map from $M^!_{2-k}(N, \psi)$ to $M^!_{k} \left ( N, \psi \cdot \left ( \frac{-1}{\cdot} \right )\frac{\psi_1}{\psi_0} \right )$ sending each
\begin{equation}\label{FourExmga}
f(z)=\sum_{n\ge -n_0}c_n q^n \in M^!_{2-k}(N, \psi)
\end{equation}
to a $f_1 \in M^!_{k} \left ( N, \psi  \left ( \frac{-1}{\cdot} \right )\frac{\psi_1}{\psi_0} \right )$ of the form
\begin{equation}\label{Q2a}f_1(z)=\sum_{n\ge -n_0}\left ( c_n \ell (n)  n^{k-1}+ \text{``lower order terms''} \right )q^n
\end{equation}
where the ``lower order terms'' are linear combinations of $\{c_{-n_0}, \dots, c_{n-1}\}$ with coefficients independent of $f$.
\end{corollary}

\section{Possible approach to Question \ref{Question13}}
We outline an approach based on Theorems \ref{DThalf} and \ref{thm:CT1} which could potentially shed light on Question \ref{Question13}. Specifically we will derive a sufficient condition for a map $\ell$ to give an affirmative answer to that question. We will first prove a functional equation that the $L$-series of a weakly holomorphic modular form must satisfy if the answer to Question \ref{Question13} is positive. We will then identify a condition that implies that functional equation. Then, by the Converse Theorem \ref{thm:CT1}, we can deduce the modularity of the function $f_1$ of \eqref{Q3}.

\subsection{Functional equations} 
Let $f$ be a weakly holomorphic cusp form of weight $2-k\in\frac12-\N$, level $N$ and  Nebentypus $\psi$ with a Fourier expansion of the form \eqref{FourExmg} and such that $$f|_{2-k}W_N=cf,$$ for some $c \in \C.$ (This is only assumed for simplicity). We  assume that the statement of Question \ref{Question13} holds and therefore, that there exists a bounded $\ell: \N \to \C$, independent of $f$, such the function 
$$f_1(z)=\sum_{n\gg-\infty} c_n \ell (n)  n^{k-1}q^n$$
belongs to $M^!_{k}(N', \psi')$, for some $N' \in \N$ and some character $\psi' \bmod N'$.
Finally, we assume that, if $f|_{2-k}W_N=cf$, for some $c \in \C$, then $f_1|_{2-k}W_{N_1}=c\lambda f_1$, for some $\lambda \in \C$, as is to be expected for a useful Bol-style operator.

The Direct Theorem \ref{DThalf} then implies that for each character $\chi$ mod $D$ with $(D, NN')=1$ and for each compactly supported $\varphi: \R^+ \to \C$ we have
\begin{align}\label{lf}
L_f(\chi,\varphi)&=i^kc\psi_D((-1)^{\frac32-k}N)\frac{N^{\frac k2}\chi(-N)\psi(D)}{\varepsilon_D}L_{f}(\overline{\chi}\psi_D,\varphi|_{k}W_N) \\
\label{lDf}
L_{f_1}(\chi, \varphi)&=i^kc\lambda \frac{\psi_D\left((-1)^{k-\frac12}N'\right)\chi(-N')\psi'(D)}{\varepsilon_D(N')^{\frac{k}{2}-1}}L_{f_1}(\overline{\chi}\psi_D,\varphi|_{2-k}W_{N'})
\end{align}
for every character $\chi$ modulo $D$ for $(D, N')=1$. 

There is a relation between $L_f(\chi,\varphi)$ and $L_{f_1}(\chi,\varphi)$: We let $h$ be a smooth function on $\R_+$ such that $h(n)=\ell(n)$ for $n \in \Z$ and we set
\begin{equation}\label{alpha}
\alpha_D(\varphi):=\scrL^{-1}\left(\left(\frac{Dp}{2\pi} \right)^{k-1}h\left(\frac{Dp}{2\pi}\right)(\scrL\varphi)(p)\right),
\end{equation}
Then, for all characters $\chi$ modulo $D$ with $(D, NN')=1.$ we have 
\begin{multline}\label{comp}
L_{f_1}(\chi,\varphi)=\sum_{n\gg-\infty}\tau_{\overline{\chi}}(n)c_n  n^{k-1}\ell(n)(\scrL\varphi)\left(\frac{2\pi n}{D}\right)\\=\sum_{n\gg-\infty}\tau_{\overline{\chi}}(n) c_n\scrL(\alpha_D(\varphi))\left(\frac{2\pi n}{D}\right)=L_f(\chi,\alpha_D(\varphi)).
\end{multline}
Thus \eqref{lDf} becomes
$$
L_{f}(\chi, \alpha_D(\varphi))=i^k c\lambda \frac{\psi_D\left((-1)^{k-\frac12}N'\right)\chi(-N')\psi'(D)}{\varepsilon_D(N')^{\frac{k}{2}-1}}L_{f}(\overline{\chi}\psi_D, \alpha_D(\varphi|_{2-k}W_{N'}))
$$
Upon applying \eqref{lf} to the left-hand side, we deduce the following proposition.
\begin{proposition}
Let $k \in \frac{1}{2}+\N,$ $N \in \N$ and a Dirichlet character $\psi$ mod $N$. 
Assume there is a linear map from $M^!_{2-k}(N, \psi)$ to $M^!_{k}(N', \psi')$, for some $N' \in \N $ and a character $\psi' \bmod N'$, sending each $f \in M^!_{2-k}(N, \psi)$ with Fourier expansion \eqref{FourExmga}
to a $f_1 \in M^!_{k}(N', \psi')$ of the form
\begin{equation}\label{Q3b}f_1(z)=\sum_{n\ge -n_0} c_n \ell (n)  n^{k-1}q^n
\end{equation}
  for an explicit, bounded map $\ell \colon \Z \to \C$ independent of $f$. 
Further assume that if $f|_{2-k}W_N=cf$, for some $c \in \C$, then $f_1|_{2-k}W_{N_1}=c\lambda f_1$, for some $\lambda \in \C$.

Then, for each $f \in S^!_{2-k}(N, \psi)$ such that $f|_{2-k}W_N=cf,$ for some $c \in \C$ and for each piece-wise smooth, compactly supported $\varphi$ on $\R_+$, we have
\begin{equation}\label{nec}L_f\left(\overline\chi\psi_D,\alpha_D(\varphi)|_kW_N\right)
=L_f\left(\overline\chi\psi_D, b \alpha_D\left(\varphi|_{2-k} W_{N'}\right)\right).
\end{equation}
where
\begin{equation}\label{b}
b:=\lambda \psi_D\left ((-1)^{2k}\frac{N'}{N} \right ) \chi\left ( \frac{N'}{N} \right ) \frac{\psi'(D)}{\psi(D)} (NN')^{-\frac{k}{2}} N'.
\end{equation}
\end{proposition}

\subsection{A sufficient condition for a positive answer to Question \ref{Question13}} 

\begin{proposition} Let $k \in \frac{1}{2}+\N,$ $\lambda \in \C,$ $N, N' \in \N$ and $\psi, \psi'$ Dirichlet characters modulo $N$ and $N'$ respectively.
Suppose that there is a $h: \R \to \C$ such that, for all smooth compactly supported $\varphi$ on $\R$ and all $\chi \bmod D$ ($(D, NN')=1$) we have,
for all $p \in \R$, 
\begin{multline}\label{SC}b \left ( \frac{Dp}{2 \pi} \right )^{k-1}h\left (\frac{Dp}{2 \pi} \right ) \mathcal L \left ( \varphi\left ( \frac{1}{Nx}\right )(Nx)^{k-2}\right )(p)\\=
\mathcal L \left ( (Nx)^{-k} \mathcal L^{-1}\left (\left ( \frac{Dp}{2 \pi} \right )^{k-1} h \left ( \frac{Dp}{2 \pi} \right )(\mathcal L \varphi)(p)\right ) \left ( \frac{1}{Nx}\right )\right )(p)
\end{multline}
where $b$ is given by \eqref{b} for some $\lambda \in \C$.
Then, if $f \in M_{2-k}^!(N, \psi)$ with Fourier expansion \eqref{FourExmga} and such that $f|_{2-k}W_N=cf,$ for some $c \in \C,$ 
then the function $f_1$ given by 
\begin{equation}\label{D^{k-1}a}f_1(z):=\sum_{n\gg-\infty}c_n  n^{k-1}h(n)q^n\end{equation}
belongs to $M_{k}^!(N', \psi')$ and $c\lambda f_1=f_1|_k W_{N'}$.
\end{proposition}
\begin{proof} We first observe that, by the definition of $\alpha_D,$ \eqref{SC} implies 
\begin{equation}\label{AlphaIntertwiningsuff}
b\alpha_D\left(\varphi |_{2-k} W_{N'}\right)=
\alpha_D\left(\varphi \right ) |_{k}W_{N},
\end{equation}
By construction, we have $L_{f_1}(\chi,\varphi)=L_f(\chi,\alpha_D(\varphi))$. Further, since $f \in M_{2-k}^!(N, \psi)$ and $f|_{2-k}W_N=cf,$ Theorem \ref{DThalf} implies \eqref{lf}. Therefore,
\begin{equation*}
L_{f_1}(\chi,\varphi)=i^k c\psi_D((-1)^{\frac32-k}N)\frac{N^{\frac k2}\chi(-N)\psi(D)}{\varepsilon_D}L_f(\overline{\chi}\psi_D, \alpha_D(\varphi)|_{k}W_N).
\end{equation*}
Then, \eqref{AlphaIntertwiningsuff} implies that this equals
\begin{equation}\label{interm}
i^k c\lambda\frac{\psi_D\left((-1)^{k-\frac12}N'\right)\chi(-N')\psi'(D)}{\varepsilon_D(N')^{\frac{k}{2}-1}}
L_f(\overline{\chi}\psi_D, \alpha_D(\varphi|_{2-k}W_{N'})).
\end{equation}
The last term of \eqref{interm} equals $L_{f_1}(\overline{\chi}\psi_D, \varphi|_{2-k}W_{N'})$ and thus we have
$$L_{f_1}(\chi,\varphi)=i^k\frac{\psi_D\left((-1)^{k-\frac12}N'\right)\chi(-N')\psi'(D)}{\varepsilon_D(N')^{\frac{k}{2}-1}}L_{c\lambda f_1}(\overline{\chi}\psi_D, \varphi|_{2-k}W_{N'}).$$ Then Theorem \ref{thm:CT1} implies that
$f_1 \in M_{k}^!(N', \psi')$ and that 
$c\lambda f_1=f_1|_k W_{N'}$.
\end{proof}
This proposition implies that ``solving the equation \eqref{SC} in $h$" would give an affirmative answer to Question \ref{Question13}.

In \cite{DLRR} (Proposition 5.5) we show that this approach works in the case of integral weight, with map $h \equiv 1$. The stumbling block to transferring this to the case of half-integral weight is that some Laplace transform identities crucial in the integral weight case do not hold or lead to infinite sums in the half-integral case. Therefore, ``solving the equation \eqref{SC} in $h$" is harder. 

For example, in the case of integral weights, the proof of the analogue of \eqref{SC} with $h \equiv 1$ hinges, in a sense, on with the simple relation $J_{-k}(z)=(-1)^kJ_k(z)$ satisfied by the J-Bessel functions when $k$ is integer.  This relation does not hold for $k \not \in \Z$, but the explicit expressions for $J_{-k}$ and $J_k$ do exhibit some similarities. Specifically, for each $n \in \N,$
we have (\cite[10.47(ii), 10.49(iii)]{NIST}):
\begin{equation}\label{JB}
\begin{aligned}
J_{n+\frac{1}{2}}(z)&=\sqrt{\frac{2}{\pi}}z^{n+\frac{1}{2}}\left (-\frac{1}{z} \frac{d}{dz} \right )^n\left ( \frac{\sin z}{z} \right ),
\\
J_{-n-\frac{1}{2}}(z)&=(-1)^n\sqrt{\frac{2}{\pi}}z^{n+\frac{1}{2}}\left (-\frac{1}{z} \frac{d}{dz} \right )^n\left ( \frac{\cos z}{z} \right ).
\end{aligned}
\end{equation}
This  more complicated pattern may, on the one hand, account for the difficulty in extending the method of proving \cite[Proposition 5.5]{DLRR} to the case of half-integral weight and, on the other, give hope that a ``solution in $h$" of \eqref{SC} may exist.

We might perhaps complete this picture by pointing out that, for $k \not \in \frac{1}{2}\Z$, there is, as far as we are aware, no recognisable relation between $J_k$ and $J_{-k}$. This could be viewed as consistent with the expectation that no weight $k$ Bol-type operator should exist for such $k$.

\thispagestyle{empty}
{\footnotesize
\nocite{*}
\bibliographystyle{amsalpha}
\bibliography{analogue}
} 

\end{document}